\documentclass[12pt, reqno]{amsart}
\usepackage{}
\usepackage{amsmath, amstext, amsbsy, amssymb, amscd}
\usepackage{amsmath}

\usepackage{bbm}
\usepackage{dsfont}
\usepackage{yfonts}
\usepackage{amsmath}
\usepackage{amssymb}
\usepackage{amsfonts}
\usepackage{mathrsfs}
\usepackage{amsxtra}
\usepackage{amscd}
\usepackage{amsthm}
\usepackage{amsfonts}
\usepackage{amssymb}
\usepackage{eucal}
\usepackage{color}
\usepackage{mathrsfs}
\usepackage[all]{xy}
\usepackage[CJKbookmarks=true]{hyperref}
\usepackage{soul}
\usepackage{shorttoc}

\newtheorem{theorem}{Theorem}
\newtheorem{rmk}{Remark}
\newtheorem{lemma}{Lemma}

\newtheorem{prop}{Proposition}

\newtheorem{defi}{Definition}

\newcommand{\cH}{\mathcal{H}}

\newcommand{\CC}{\ensuremath{\mathbb{C}}}
\makeatother

{\vskip-\lastskip\medskip
  \noindent
  {\em #1.}\enspace
  }%
{\qed\par\medskip
  }

\begin{document}

\title[Equivariant homology theory and twisted Yangian]{Equivariant homology theory and twisted Yangian}

\author[Z. Dong, H. Ma ]{Zhijie Dong, Haitao Ma }
\address{Institute for Advanced Study in Mathematics of HIT, Harbin, 150001, China}
\address{College of mathematics science, Harbin Engineering University, Harbin, 150001, China.}
\email{dongmouren@gmail.com (Zhijie Dong)}
\email{hmamath@hrbeu.edu.cn (Haitao Ma)}
\address{}
\thanks{}
\maketitle
\begin{abstract}
We study the convolution algebra $H^{G\times \CC^{*}}_{*}(Z)$ of $G$-equivariant homology group on the Steinberg variety of type B/C and define an algebra $\widetilde{Y}$ that maps to  $H^{G\times \CC^{*}}_{*}(Z)$.
The Drinfeld new realization of the twisted Yangian associated to symmetric pairs is a quotient of $\widetilde{Y}$.  
We also study the $G$-equivariant case
and prove that the twisted Yangian is the deformation of the twisted current algebra.
\end{abstract}

\section{Introduction}

In the paper \cite{kamnitzer2015highest}, it was proved that the ring of functions on the
$T$-fixed point subscheme of affine Grassmannian slice is isomorphic to the cohomology ring of quiver variety,\\
\begin{equation}
\mathcal{O}(\overline{Gr^{\lambda}_{\mu}}^{T})\cong H^{*}(\mathfrak{M}(V,W)). \label{1}
\end{equation}
This is a special case of Hikita conjecture which says that given an algebraic group $G$ and a representation $V$, the function ring on the $T$-fixed point subscheme of the Coulomb branch $\mathfrak{M}_{C}$ is isomorphic to the cohomology ring of the Higgs branch $\mathfrak{M}_{H}$.
The key point in the proof is that there exists an Yangian action on the homology groups $\bigoplus\limits_{V} H^{*}_{G(W)\times \CC^{*}}(\mathfrak{M}(V,W)$.
A special case of (1) is to take  $\mathfrak{M}(V,W)$ to be cotangent bundle of flag variety of type A.
This motivates us to study (1) for flag variety of other types. 
In this case, currently we do not know what is the left hand side but we can still study what is the algebra that acting on certain sum  of $H^{*}_{G\times \CC^{*}}(T^{*}(G/P))$ over $P$.

On the other hand,
 the right analog of other symmetric types to type A flag variety is the quiver variety discovered by Nakajima \cite{nakajima1994instantons}.
For symmetrizable types, it is still not clear what should be the geometric object playing the role of quiver variety
(see \cite{nakajima2019coulomb} for recent development).
Our work is somehow a drawback in this direction where we consider type B/C flag variety which does not correspond to type C/B Lie algebra.

In the paper \cite{fan2019equivariant},  the equivariant K-theory on the Steinberg variety of type B/C was studied.   There is a homomorphism from the $\imath$- quantum group associated to the symmetric pair  to the  equivariant K-group.
In this paper, we study the convolution algebra of the homology case. The main result is the following.
We construct a map from an algebra $\widetilde{Y}$ which we call pre-twisted Yangian to the $G$-equivariant convolution algebra of the  Steinberg variety. Let $\sigma$ be the involution of $\mathfrak{gl}_n$ induced by diagram automorphism of the corresponding Dynkin diagram. People studied the twisted polynomial current Lie algebra $\mathfrak{gl}_n[x]^{\sigma}$ and its 
deformation, which is called twisted Yangian. 
In our case, we prove that there exists a map 
$$ U(\mathfrak{sl}_{2n+1}[x]^{\theta})\xrightarrow[]{} H^{G_B}_{*}(Z_B)(\text{ \ resp. \ } H^{G_C}_{*}(Z_C) ).$$ for another involution $\theta$ of $\mathfrak{sl}_{2n+1}$. 
We expect that there exist twisted Yangians $Y^{B}_{\hbar}$ and $Y^{C}_{\hbar}$ such that$$
 \xymatrix{
 &Y^{B}_{\hbar} \ar[r]^{} \ar[dd]^{\hbar=0} & H^{G_{C}\times \CC^{*}}(Z_C) \ar[d]^{\hbar=0}\\
 && H^{G_{C}}(Z_C)\\
 \widetilde{Y}  \ar@{->>}[ruu]{}  \ar@{->>}[rdd]{} & U(\mathfrak{sl}_{2n+1}[x]^{\theta}) \ar[ru]{} \ar[rd]{} & \\
 && H^{G_{B}}(Z_B)\\
 &Y^{C}_{\hbar} \ar[r]^{} \ar[uu]^{\hbar=0} & H^{G_{B}\times \CC^{*}}(Z_B) \ar[u]^{\hbar=0}.\\
 }$$
We can not find all the relations of $Y^{B}_{h}$ and  $Y^{C}_{h}$. We denote the algebra that has the relations as many as possible in our capability by $\widetilde{Y}$ and call it pre-twisted Yangian.

The approach we use is completely standard as \cite{varagnolo2000quiver}.
Let me summarize it again here.
We use the action of $H^{G}_{*}(Z)$ on $ H^{G}(\bigsqcup G/P_{\underline{\nu}})$ (defined in section 4). We use localization theorem to reduce the computation of intersection of homology classes to Euler classes and intersection of point classes. We prove that the action is faithful so that we could check relations through it.

Lastly, we go back to our motivation and mention some future work.
Flag variety of type B/C has an realization as the core of $\sigma$-quiver variety \cite{li2019quiver} so we should study the Steinberg type variety of $\sigma$-quiver variety and expect that the algebra that maps to it is the same
as $Y_{\hbar}^{C}$ (resp. $Y_{\hbar}^{B}$). Also, we should try to understand if it could fit into the framework of Coulomb branch and the Hikita conjecture in this case.

The organization of the paper is as follows.
In the second section, we introduce the pre-twisted Yangian $\widetilde{Y}$.
In the third section, we  introduce some basic concepts of convolution algebra.
In the fourth section, we 
introduce type C flag variety and Steinberg variety.
In the fifth section, we define  generators.
In the sixth section, we compute the action on the zero fiber and prove our main theorem.
In the seventh section, we study type B case.
In the last section, we prove there are maps from $ U(\mathfrak{sl}_{2n+1}[x]^{\theta})$ to $H^{G_B}_{*}(Z_B)(\text{\ resp.\ } H^{G_C}_{*}(Z_C) ).$

\section{Pre-twisted Yangian $\widetilde{Y}$}
\begin{defi}
The pre-twisted Yangian $\widetilde{Y}$ is the associative $\mathbb{C}[\hbar]$-algebra generated by $\mathbf{e}_{i,r}$, $\mathbf{f}_{i,r}$,$ \mathbf{h}_{i,r}$($i \in [1,n], r \in \mathbb{N}$ ) with the following defining relations
  \allowdisplaybreaks
\begin{eqnarray*}
&&[\mathbf{h}_{i,r},\mathbf{h}_{j,s}] = 0,\\
&&[\mathbf{h}_{i,0},\mathbf{e}_{j,r}]=(2\delta_{i, j} -\delta_{i,j+1}- \delta_{i,j-1}  + \delta_{i,n}\delta_{j,n})\mathbf{e}_{j,r},\\
&&[\mathbf{h}_{i,0},\mathbf{f}_{j,r}]=(-2\delta_{i, j} +\delta_{i,j+1}+ \delta_{i,j-1}  - \delta_{i,n}\delta_{j,n})\mathbf{f}_{j,r}, \\
&&2[\mathbf{h}_{i,r+1},\mathbf{e}_{j,s}] - 2[\mathbf{h}_{i,r},\mathbf{e}_{j,s+1}] = \hbar (2\delta_{i, j} -\delta_{i,j+1}- \delta_{i,j-1})(\mathbf{h}_{i,r}\mathbf{e}_{j,s} +\mathbf{e}_{j,s}\mathbf{h}_{i,r} ), \ {\rm if} \ j \neq n,\\
(R1)  &&2[\mathbf{h}_{n,r},\mathbf{e}_{n,s+2}] - \hbar^2[\mathbf{h}_{n,r},\mathbf{e}_{n,s}] - 4[\mathbf{h}_{n,r+1},\mathbf{e}_{n,s+1}] + 2[\mathbf{h}_{n,r+2},\mathbf{e}_{n,s}] \\
  && \quad \quad \quad= \hbar(\mathbf{h}_{n,r+1}\mathbf{e}_{n,s} + \mathbf{e}_{n,s}\mathbf{h}_{n,r+1}) - \hbar(\mathbf{h}_{n,r}\mathbf{e}_{n,s+1} + \mathbf{e}_{n,s+1}\mathbf{h}_{n,r}),\\
&&2[\mathbf{h}_{i,r+1},\mathbf{f}_{j,s}] - 2[\mathbf{h}_{i,r},\mathbf{f}_{j,s+1}] = -\hbar (2\delta_{i, j} -\delta_{i,j+1}- \delta_{i,j-1})(\mathbf{h}_{i,r}\mathbf{f}_{j,s} +\mathbf{f}_{j,s}\mathbf{h}_{i,r} ),\ {\rm if} \ j \neq n,\\
(R2)&&2[\mathbf{h}_{n,r},\mathbf{f}_{n,s+2}] - \hbar^2[\mathbf{h}_{n,r},\mathbf{f}_{n,s}] - 4[\mathbf{h}_{n,r+1},\mathbf{f}_{n,s+1}] + 2[\mathbf{h}_{n,r+2},\mathbf{f}_{n,s}] \\
  && \quad \quad \quad = -(\hbar(\mathbf{h}_{n,r+1}\mathbf{e}_{f,s} + \mathbf{f}_{n,s}\mathbf{h}_{n,r+1}) - \hbar(\mathbf{h}_{n,r}\mathbf{f}_{n,s+1} + \mathbf{f}_{n,s+1}\mathbf{h}_{n,r})),\\
&&[\mathbf{e}_{i,r},\mathbf{f}_{j,s}] = \delta_{ij} \mathbf{h}_{r+s},{\rm if} \ i,j \neq n,\\
(R3)&&2[\mathbf{e}_{i,r+1},\mathbf{e}_{j,s}] - 2[\mathbf{e}_{i,r},\mathbf{e}_{j,s+1}] = \hbar(2\delta_{i, j} -\delta_{i,j+1}- \delta_{i,j-1})(\mathbf{e}_{i,r}\mathbf{e}_{j,s} + \mathbf{e}_{j,s} \mathbf{e}_{i,r}),\ \rm{if} \ (i,j) \neq (n, n),\\
(R4)&&2[\mathbf{f}_{i,r+1},\mathbf{f}_{j,s}] - 2[\mathbf{f}_{i,r},\mathbf{f}_{j,s+1}] = - \hbar(2\delta_{i, j} -\delta_{i,j+1}- \delta_{i,j-1})(\mathbf{f}_{i,r}\mathbf{f}_{j,s} + \mathbf{f}_{j,s} \mathbf{f}_{i,r}),  \\
&&\sum\limits_{w \in S_2}[\mathbf{e}_{i,r_{w_1}},[\mathbf{e}_{i,r_{w_2}}, \mathbf{e}_{j,s}]] = 0,\ \rm{if} \ |i - j| = 1,\\
&&\sum\limits_{w \in S_2}[\mathbf{f}_{i,r_{w_1}},[\mathbf{f}_{i,r_{w_2}}, \mathbf{f}_{j,s}]] = 0,\ \rm{if} \ |i - j| = 1,\\
(R5)&&[\mathbf{e}_{n,0},[\mathbf{e}_{n,0}, \mathbf{f}_{n,0}]] = -4\mathbf{e}_{n,0},\\
(R6)&&[\mathbf{f}_{n,0},[\mathbf{f}_{n,0}, \mathbf{e}_{n,0}]] = -4\mathbf{f}_{n,0}.
\end{eqnarray*}
\end{defi}
\section{Convolution algebra}
We use the notation $H_{*}$ for Borel-Moore homology
and when we say homology we mean Borel-Moore homology.
Denote by $H_{*}^{G}$ the $G$-equivariant homology.
We use $EG_{N}$ to approximate the classifying space $EG$ , where $G$ acts on $EG$ freely and  $EG$ contracts to a point. We define the $H^{G}(X)$ as $H(X\times_{G} EG_{N})$ for N sufficiently large keeping the codimension unchanged. Namely, $H^{G}_{*}(X)=H_{dim(EG_{N}-dimG+*)}(X\times_{G} EG_{N})$.
The definition is independent of the choice of $EG_{N}$ by standard argument.

We summarize how $H^{G}_{*}(Z)$ is endowed with an algebraic structure, see\cite{chriss2009representation} for detail.
We state the general setting for equivariant case.
Let $X_1, X_2, X_3$ be $G$-varieties.
Let $Z_{12}$ be $G$-invariant subvariety in $X_1\times X_2$ and $Z_{23}$ be $G$-invariant subvariety in $X_2\times X_3$.
Let $G$ act on the product of $X_i$ diagonally.
Denote the projection $X_i\times X_j \times X_k \times_{G} EG \xrightarrow{} X_i\times X_j \times_{G} EG$ by $\pi_{ij}$.
Given two homology classes $z_{12}\in H^{G}_{*}(Z_{12})$ and $z_{23}\in H^{G}_{*}(Z_{23})$, define the convolution product \\
$$z_{12}\star z_{23}=(\pi_{13})_{*}(\pi^{*}_{12}(z_{12})\cap \pi^{*}_{23}(z_{23})),$$
where lower star means pushforward which is defined for proper map and upper star means pullback which is defined for bundle map.
Define the set theoretic image of  $\pi^{*}_{12}(Z_{12})\cap \pi^{*}_{23}(Z_{23})$ under $\pi_{23}$  by $Z_{13}$.
We see that $z_{12}*z_{23}\in H^{G}(Z_{13})$ hence we have
$H^{G}_{*}(Z_{12})\times H^{G}_{*}(Z_{23})\xrightarrow{\star} H^{G}_{*}(Z_{13})$ and it can be shown that  this convolution product is associative.
Given a proper map $X\xrightarrow{\pi} Y$,
we set $X_1=X_2=X_3=X$ and $Z_{12}=Z_{23}=Z=X\times_{Y} X$.
We have $Z_{13}=Z$ and $H^{G}_{*}(Z)\times H^{G}_{*}(Z) \xrightarrow{\star} H^{G}_{*}(Z)$ so $H^{G}_{*}(Z)$ has an associative algebra structure.
Set $X_1=X_2=X, X_3=pt$ , $Z_{12}=Z=X\times_{Y} X$, $Z_{23}=X\times_{Y} pt$, where $pt$ maps to $y\in Y$ in the second map of the fiber product then $Z_{13}=Z_{23}=\pi^{-1}(y)$ and the convolution $H^{G}_{*}(Z)\times H^{G}_{*}(\pi^{-1}(y)) \xrightarrow{\star} H^{G}_{*}(\pi^{-1}(y))$ makes $H^{G}_{*}(\pi^{-1}(y))$ a $H^{G}_{*}(Z)$ module.

\section{ype C flag variety and Steinberg variety} \label{steinberg variety of type c}
Let $V$ be a $2d$ dimensional vector space  over complex numbers.
We fix a nondegenerate skew-symmetric bilinear form $(\  , \ )$. For a subspace $W\subset V$, let $W^{\perp}=\{x\in V\ |\ (x,w)=0\  \text{for\ any}\ w\in W\}$.
Denote by $Fl_{2n}$ the $2n$-step partial flag variety in $V$. A partial flag $fl=0=V_0\subset V_{1}\subset V_{2} \cdots \subset V_{n} \cdots \subset V_{2n}\subset V_{2n+1}=V$ is a nested sequence of subspaces in $V$ subject to the relations $V_{n+i}=V_{n-i+1}^{\perp}$ for $1\leq i \leq n$.
For a flag $fl$, let $\nu_{i}=dim V_{i}-dim V_{i-1}$ and $\underline{\nu}=(\nu_1,\cdots , \nu_{2n+1})
$ be the dimension vector. Note that the sequence $\nu_{1},\cdots ,\nu_{2n+1}$ forms a partition of $2d$ satisfying $\nu_{i}=\nu_{2n+2-i}$ for $1\leq i \leq n$. Denote by $Fl_{\underline{\nu}}$ the flag variety in $V$ of dimension vector $\underline{\nu}$.  Namely $Fl_{\underline{\nu}}$ consists of flag $0\subset V_{1}\subset V_{2} \cdots V_{n}\subset V_{n+1} \subset \cdots \subset V_{2n}\subset V$ such that $dim V_{i}-dim V_{i-1}=\nu_{i}$. Hence $Fl_{n}$ is the disjoint union of $Fl_{\underline{\nu}}$ over all dimension vectors $\underline{\nu}$. The cotangent bundle of $Fl_{\underline{\nu}}$ is the incident variety consisting of  all pairs $(fl,x)$ , where $x\in Sp(V)$ such that $x V_{i+1} \subset V_{i}$ for $1\leq i \leq 2n+1$. Let $\pi: X \xrightarrow{\pi} Y, \ (fl,x)\mapsto x.$
Let $G=Sp(V)$ and
 $\Tilde{G}=Sp(V)\times \CC^{*}$, where $\CC^{*}$ is the multiplicative group. 
The group $Sp(V)$ acts $V$ so induces a natural action on the flag variety and hence a natural action on its cotangent bundle $T^{*}Fl_{\underline{\nu}}$. Explicitly, 
$$(g,c)(0\subset V_1 \cdots \subset V),x)=(0\subset gV_1 \cdots \subset V, c^{2}gxg^{-1}).$$
Denote the fiber product of $X\times_{Y} X$ by $Z$, which we will refer to Steinberg variety of type C.
\section{Generators}
 We need some preparation before defining the generators.
 First we write $Z$ as disjoint union of irreducible varieties.
 Since $X=\sqcup_{\underline{\nu}} T^*Fl_{\underline{\nu}}$, we have
 $$ Z=X\times _{Y} X = \sqcup_{\underline{\nu}} T^*Fl_{\underline{\nu}}\times_{Y} \sqcup_{\underline{\nu}} T^*Fl_{\underline{\nu}}=\sqcup_{\underline{\nu},\underline{{\nu}^{\prime}}} T^*Fl_{\underline{\nu}} \times_{Y} T^*Fl_{\underline{\nu^{\prime}}}.$$
 We denote by $Z_{\underline{\nu},\underline{\nu}^{\prime}}$  the last term.
In term of incident variety, $Z_{\underline{\nu},\underline{\nu}^{\prime}}$ consists of triples ${(fl,fl',x)}$ such that
$fl \in Fl_{\underline{\nu}}$, $fl' \in Fl_{\underline{\nu'}}$, $x\in Sp(V)$,  $xV_{i+1}\subset V_{i}$ , $xV_{i+1}^{\prime}\subset V_{i}^{\prime}.$
Let $1_i=(0,\cdots,0,1,-1,\cdots,0)$ where $1$ is in the $i$-th position.
 we now identify the partial flag variety as homogeneous space of $Sp(V)$.  Choose a basis coordinate to the pairing $(\  , \ )$ such that $(e_i,e_{j})=\delta_{j-i,n}$ for $1\leq i\leq j \leq n$.
We fix a full base flag $fl^0=0\subset \langle e_1 \rangle \subset \langle e_1,e_2  \rangle \cdots \langle e_1,\cdots,e_n \rangle \subset \langle e_1,\cdots,e_{n+1} \rangle \subset \cdots 
 \subset V$.
For a dimension vector $\underline{\nu}$, denote by $fl^0_{\underline{\nu}}$ the partial flag by remembering the corresponding subspaces in the full base flag $fl^0$.
Denote the stabilizer of $fl^0_{\underline{\nu}}$ by $P_{\underline{\nu}}$ so we have
$Fl_{\underline{\nu}}\cong G/P_{\underline{\nu}}$.
Let $\pi: Z_{\nu,\nu^{\prime}}\xrightarrow{} Fl_{\underline{\nu}}\times Fl_{\underline{\nu^{\prime}}}$, $(fl,fl',x)\mapsto (fl,fl')$ and this is a bundle map.

By Bruhat decomposition, we have the following lemma.

\begin{lemma}\label{orbit sturcture of type c}

The Steinberg variety $Z_{\underline{\nu},\underline{\nu^{\prime}}}$ and $Fl_{\underline{\nu}}\times Fl_{\underline{\nu^{\prime}}}$ decompose into
G-orbits indexed by $w\in W_{P_{\underline{\nu}}} \backslash W/ W_{P_{\underline{\nu}}^{\prime}} $. Let $S_{w}=P_{\underline{\nu}}w P_{\underline{\nu^{\prime}}}/P_{\underline{\nu^{\prime}}}\subset G/P_{\underline{\nu}'}$ be the subvariety in the partial flag variety $G/P_{\underline{\nu}'}$ and
let $\Omega_w=\{(gP_{\underline{\nu}},gS_{w})|g\in G\} \subset Fl_{\underline{\nu}}\times Fl_{\underline{\nu^{\prime}}}.$
We have
$$Fl_{\underline{\nu}}\times  Fl_{\underline{\nu^{\prime}}}=
G/P_{\underline{\nu}}\times  G/P_{\underline{\nu^{\prime}}}=
\cup_{w\in  W_{P_{\nu}} \backslash W/ W_{P_{\nu}^{\prime}} }\Omega_w.$$
Denote $\pi^{-1}(\Omega_w)$ by $Z^{w}_{\underline{\nu},\underline{\nu^{\prime}}}$.
We have
$$ Z_{\nu,\nu^{\prime}}=\bigcup_{{w\in  W_{P_{\nu}}} \backslash W/ W_{P_{\nu}^{\prime}} }Z^{w}_{\underline{\nu},\underline{\nu^{\prime}}}.$$
\end{lemma}

\begin{lemma}
Let $e$ be the identity element $e\in W$.
 The orbit $Z^{e}_{\underline{\nu},\underline{\nu}+1_i}$ (resp.  $Z^{e}_{\underline{\nu},\underline{\nu}-1_i}$) is the incident variety of triples $(fl,fl',x)$ in $Z_{\underline{\nu},\underline{\nu}+1_i}$ (resp.  $Z_{\underline{\nu},\underline{\nu}-1_i}$) with extra conditions $dim(V'_j/V_j)=\delta_{ij}$ (resp. $dim(V_j/V'_j)=\delta_{ij}$ ) for $1\leq j \leq 2n+1$.
 The image 
 $\pi(Z_{\underline{\nu},\underline{\nu}+1_i})$ (resp.  $\pi(Z_{\underline{\nu},\underline{\nu}-1_i})$)  consists of all the  pairs $(fl_1,fl_2)$ with extra conditions $dim(V'_j/V_j)=\delta_{ij}$ (resp. $dim(V_j/V'_j)=\delta_{ij}$ ) for $1\leq j \leq 2n+1$, which is $G/P_{\underline{\nu}}\cap P_{ \underline{\nu}+1_i}$

\end{lemma}

\begin{proof}
The two base flags of dimension $\underline{\nu}$ and $\underline{\nu}+1_i$ satisfy the extra condition and the group action preserves it.
\end{proof}

Let $\overline{\nu_i}=\Sigma_{1\leq j\leq i} \nu_j$. We have a natural line bundle $L^{+}_i$ (resp. $L^{-}_i$)
where the fiber over $(fl_1,fl_2,x)$ is the line $V_i^{\prime}/V_i$ (resp $V_i/V_i^{\prime}$) . Denote the first $G$-equivariant Chern class of the line bundle by $c_1(L^{+}_{i})\in H^{*}_{G}(Z^{e}_{\underline{\nu},\underline{\nu}+1_{i}})$. (resp. $c_1(L^{-}_{i})\in H^{*}_{G} (Z^{e}_{\underline{\nu},\underline{\nu}-1_{i}})$)
We understand cohomology class as an operator acting on homology class by cap product and use $\bullet$ for this action.
We can now define generators $\mathcal{E}_{i,r}$ and $\mathcal{F}_{i,r}$.

$$  \mathcal{E}_{i,r} = \sum\limits_{\underline{\nu}} (-1)^{\nu_{i+1}+\delta_{n,i}}c_1(L^{+}_i)\bullet[Z^{e}_{\underline{\nu},\underline{\nu}+1_{i}}]$$ 
$$  \mathcal{F}_{i,r} = \sum\limits_{\underline{\nu}} (-1)^{\nu_i}c_1(L^{-}_i)\bullet[Z^{e}_{\underline{\nu},\underline{\nu}-1_{i}}].$$
In order to define the other generators we need to prepare more notations. Let $L(q^{m})\in K^{G}(Z^{e}_{\underline{\nu},{\underline{\nu}}})$ be the trivial line bundle with $c\in \CC^{*}$ acting by scaling by $c^m$.
Note that $Z^{e}_{\underline{\nu},{\underline{\nu}}}$ is a vector bundle over the diagonal part of  $Fl_{\underline{\nu}}\times Fl_{\underline{\nu}}$.
Denote by $\mathcal{V}_{i}$ the tautological bundle over $Z^{e}_{\underline{\nu},{\underline{\nu}}}$ where the fiber over $(fl,fl,x)$ is $V_i$ for $1\leq i\leq n$. 
Let $F_{i,\underline{\nu}}$ in the $G$-equivariant $K$-group be $-(1+q^{-2})V_{i}+q^{-1}(V_{i+1}+V_{i-1}).$ Denote by $\lambda$ the $G$-equivariant Chern polynomial map from the $G$-equivariant $K$-group $K^{G}(X)$ to
$H^{G}_{*}(X)[z]$, the polynomial ring with coefficients in the $G$-equivariant homology group.
Let $\hbar=c_1(L(q^2))$.
We now define the $\mathcal{H}_{i,r,\underline{\nu}}$ to be the coefficient of $\hbar z^{-r-1}$ in the formal expansion of
$$ \frac{\lambda_{-1/z}(F_{i,\underline{\nu}})}{\lambda_{-1/z}(q^2 F_{i,\underline{\nu}})},$$ 
and  $\mathcal{H}_{i,r}=\sum_{\overline{\nu} }\mathcal{H}_{i,r,\underline{\nu}}$. 

\section{The action on the zero fiber}
we will apply localization theory.

Let $T$ be a torus acting on a variety $X$ with isolated fixed points $X^{T}=\sqcup x_i$. The embedding $i$ of $X^{T}$ to $X$ induces a map on $T$-equivariant homology group\\
$$ H^{T}(X^{T})\xrightarrow{i_{*}}H^{T}(X).$$
The map   $X\times_{T} ET\xrightarrow{} ET/T$
defines $H^{T}_{*}(X)$ (resp. $H^{*}_{T}(X)$)  as $H^{*}_{T}(pt)$ module by pulling back and take cup (resp. cap) product.
\begin{lemma}
Denote by $S_{T}$ the symmetric algebra of characters of $X^{*}(T)\otimes \CC$ , where $X^{*}(T)$ is the character lattice of $T$. We have\\
$$S_{T} \cong H^{*}_{T}(pt).$$
as vector space isomorphism with  addition mapping to multiplication and vice versa.
\end{lemma}
\begin{proof}
Given a character $\chi: T\xrightarrow{} \CC^{*}$, we have a associated line bundle $$
\CC \times_{T} ET \xrightarrow{} ET/T$$
where $T$ acts on $\CC$ by multiplying $\chi(t)$, for $t\in T$.
Let $c$ be the first Chern class of this line bundle. By property of Chern class, we see that product of characters becomes addition in cohomology classes.
For sum of characters, we define $c$ to be the top Chern class of the direct sum of corresponding line bundles and this extends to $S_(T)$.

\end{proof}

Let $Fr_{T}$ be the fraction field of $S_{T}$.
Let $\cH^{T}(\bullet)=H^{T}(\bullet)\otimes_{S_{T}} Fr_{T}$.
We have $H^{T}_{*}(X^T)=\oplus_{x_i} H^{T}_{*}(x_i)=\oplus_{x_i} S_T$ and  denoted by  $(s_i)_{x_i}, s_i\in S_T$ an element in it,  similarly for $\cH^{T}_{*}(X^T)$.
We still denote by $i_{*}$ the pushforward between localized homology groups.
Denote by $[\bullet]^{A}$ the $A$-equivariant fundamental class, where $A$ is $T$ or $G$ in this paper. We often omit $A$ when it is clear from the context.
\begin{theorem}\cite[Theorem 13,Corollary 15]{brion1998equivariant}
The $S_T$ linear map 
$$i_{*}: H^{T}_{*}(X^{T})\xrightarrow[]{} H^{T}_{*}(X)$$
becomes an isomorphism after inverting finitely many non trivial characters. Hence, the map $i_{*}:\cH^{T}_{*}(X^{T})\xrightarrow{} \cH^{T}_{*}(X)$
is an isomorphism. Let $Eu_x(X)$ be the $T$-equivariant Euler class at point $x\in X^T$. Then under the above isomorphism, the preimage of $[X]$ is $$ (\frac{1}{Eu_x(X)})_{{x\in X^{T}}}.$$ Moreover\footnote{In order to do convolution product, we have to use $i_{*}$ again.},
$$[X]=i_{*}i_{*}^{-1}([X])=
i_{*}((\frac{1}{Eu_x{X}})_{x\in X^{T}})
=\Sigma_{x\in X^{T}}\frac{1}{Eu_x{X}}[x].
$$
\end{theorem}
\begin{rmk}
  We just rewrite the fundamental class in the larger space $\cH^{T}(X)$ as sum of point classes with coefficients in $Fr_{T}$. Denote  $r([X])= \Sigma_{x\in X^{T}} \frac{1}{Eu(x,X)}[x] $.
\end{rmk}
To apply the above theorem we have to
relate $T$-equivariant homology with $G$-equivariant homology.
\begin{theorem} [\cite{brion1998equivariant}, proposition 1 (i)]
Let $X$ be a $G$-space. Let $T$ be the Cartan of $G$ with normalizer $N$ and Weyl group $W=N/T$.
Then $W$ acts on $H^{*}_{T}(X)$ and we have isomorphism\\
$$H^{*}_{T}(X)^{W}\cong H^{*}_{G}(X).$$
\end{theorem}
\begin{rmk}
Denote the map
$EG \times_{T} X \xrightarrow{} EG \times_{G} X$ by $i_{T,G}$. From the proof of the above theorem we see that $H^{*}_{G}(X)\xrightarrow{i^{*}_{T,G}} H^{*}_{T}(X)$ is injective. In particular, when $X$ is a point, $H^{*}_{T}=S_T$, the Weyl group $W$ acts on $T$ and induces an action on $S_{T}$.
\end{rmk}
We have the following commuting diagram
\begin{align}
\xymatrix{
H^{G}_{*}(Z)\times H^{G}_{*}(Fl_{n})\ar[d]^{i^*_{T,G}\times i^*_{T,G}} \ar[r]^{\star} &H^{G}_{*}(Fl_{n})\ar[d]^{i^*_{T,G}}\\
H^{T}_{*}(Z)\times H^{T}_{*}(Fl_{n})\ar[r]^{\star}\ar[d]^{r\times r} &H^{T}_{*}(Fl_{n}) \ar[d]^{r}\\
\cH^{T}_{*}(Z)\times \cH^{T}_{*}(Fl_{n})\ar[r]^{\star} &\cH^{T}_{*}(Fl_{n}).}
\end{align}

In order to compute the first horizontal map, we use the vertical map to reduce the computation to the bottom map so we have to understand to vertical map.
By the expression of $\mathcal{E}_{ir}$ we only have to compute on each component $Z^{e}_{\underline{\nu},\underline{\nu}+1_{i}}$.
The case of $\mathcal{F}_{ir}$ is similar so we focus on  $\mathcal{E}_{ir}$ case afterwards.
The action is zero if the class is not in the components  $Fl_{\underline{\nu}+1_{i}}$ and $\mathcal{E}_{ir}
 (H^{G}_{*}(Fl_{\underline{\nu}+1_{i}})) \subset H^{G}_{*}(Fl_{\underline{\nu}})$.
We could rewrite as follows
$$
\xymatrix{
H^{G}_{*}(Z^{e}_{\underline{\nu},\underline{\nu}+1_{i}})\times H^{G}_{*}(Fl_{\underline{\nu}+1_{i}})\ar[d]^{i^*_{T,G}\times i^*_{T,G}} \ar[r]^-{\star}  &H^{G}_{*}(Fl_{\underline{\nu}})\ar[d]^{i^*_{T,G}}\\
H^{T}_{*}(Z^{e}_{\underline{\nu},\underline{\nu}+1_{i}})\times H^{T}_{*}(Fl_{\underline{\nu}+1_{i}})\ar[r]^-\star\ar[d]^{r} &H^{T}_{*}(Fl_{\underline{\nu}}) \ar[d]^{r}\\
\cH^{T}_{*}(Z^{e}_{\underline{\nu},\underline{\nu}+1_{i}})\times \cH^{T}_{*}(Fl_{\underline{\nu}+1_{i}})\ar[r]^-\star &\cH^{T}_{*}(Fl_{\underline{\nu}}). \\
}
$$


We denote the map $H^{*}_{T}(X) \xrightarrow{l} H^{*}_{T}(X^T)$,  where $X=Z^{e}_{\underline{\nu},\underline{\nu}+1_{i}}, Fl_{\underline{\nu}}$.
We have $$H^{*}_{G}(X)=H^{*}(pt/P)=H^{*}(pt/P)=S_{T}^{W_P},$$
$$H^{*}_{T}(X^T)=H^{*}_{T}(\sqcup_{w\in W/W_P} p_{w})=\oplus_{w\in W/W_P} H^{*}_{T}(p_w)=\oplus_{w\in W/W_P} (S_{T})_{w}.$$
\begin{lemma}
Take $f\in H^{*}_{G}(X)$,
We have $(l\circ i^{*}_{T,G})f= (w(f))_{w\in W/W_P},$ where $H^{*}_{G}(X)=S_{T}^{W_P}\subset S_{T}$ and the action of $W_P$ on it is explained in remark 2.
\end{lemma}
\begin{proof}
We could assume $f=c(\chi)$ , where $\chi: T\xrightarrow{} \CC$ is a character of $T$ and let 
$G\times_P L$ is the associated line bundle of it.
We have the diagram
$$
\xymatrix{
EG\times_{T} (g_w \times_{P} L) \ar[r] \ar[d] & EG \times_{T}(G\times_{P} L)  \ar[r] \ar[d] &  EG \times_{G} (G\times_{P} L) \ar[d]\\
EG \times_{T} p_{w}\ar[r] &EG\times_{T} (G/P) \ar[r] &   EG\times_{G}(G/P)},
$$
where $g_w\in G$ is a lift of $p_w$.
We need to know how $T$ acts on $g_w \times_{P} L$.
The action of $t\in T$ on an element $(\overline{g_w,l})\in g_w \times_{P} L$ is given by $t(\overline{g_w,l})=(\overline{tg_w,l})=(\overline{g_wt^{g_w},l})=(\overline{w,t^{g_w}l})$, where $t^{g_w}=g_w^{-1}tg_w$.
This means $L\times_{T} EG$ is the line bundle
associated to the character\footnote{Here the action of $w$ is the inverse of conjugation but we will  use conjugation later since we will take sum over $W_P$ later so it does not matter} $w(\chi)$.
\end{proof}

We can now compute the composition of $r\circ i^{*}_{T,G}$ for any
class $f\in H^{G}_{*}(X)$.
\begin{theorem}
For $f\in H^{G}_{*}(X)$, we have
$$(r\circ i^{*}_{T,G})f= \Sigma_{w\in X^{T}} w(f) (Eu(w,X))^{-1} [x_w].$$
\end{theorem}

\begin{proof}

Using the following commutative diagram,
$$
\xymatrix{
H^{*}_{G}(X)\times H^{G}_{*}(X)\ar[d]^{i^{*}_{T.G}\times i^{*}_{T.G}} \ar[r]^-\cap &H^{G}_{*}(X)\ar[d]^{i^{*}_{T.G}}\\
H^{*}_{T}(X)\times H^{T}_{*}(X)\ar[r]^-\cap \ar[d]^{l\times r} &H^{T}_{*}(X) \ar[d]^r \\
H^{*}_{T}(X^T)\times \cH^{T}_{*}(X)\ar[r]^-\cap  &\cH^{T}_{*}(X),  }
$$
 where the second factor we plug in the fundamental class of $X$.
 We can reduce the computation of the right hand side to the left hand side since
any class in $H^{G}_{*}(X)$ can be written as $f[X]$ where $f\in H^{*}_{G}(X)$ by Poincare duality.

Then the theorem follows from Lemma 4 and Theorem 1.

\end{proof}

 To compute the convolution product of two point classes, we need equivariant clean intersection formula.
 \begin{lemma}\cite{przezdziecki2015geometric}
 Let $X$ be a smooth $G$-variety and $Y_1,Y_2$ be two close $G$-subvarieties.
 Let $Y$ be the set-theoretic intersection of $Y_1$ and $Y_2$.
 Let $\mathcal{N}$ be the bundle
 $$\mathcal{N}=TX|_{Y}/(TX|_{Y_1}\oplus TX|_{Y_2}),$$ then we have
 $$[Y_1]\cap [Y_2]= eu_{G}(\mathcal{N})[Y].$$
 
 \end{lemma}
 Now we can compute the convolution product of two point classes.
 \begin{theorem}\cite{przezdziecki2015geometric}
Let $Z_{12}=(a_1,a_2)\in M_1\times M_2$ and $Z_{23}=(a_2,a_3)$ in $M_2\times M_3$.  Then we have $[Z_{12}]^{T}\star [Z_{23}]^{T}=eu_{a_2,M_2} [Z_{13}]^{T} $, where $Z_{13}=(a_1,a_3)$.
 \end{theorem}
 \begin{proof}
 We repeat the proof for convenience of the readers.
 By definition,
 $$[Z_{12}]^{T}\star [Z_{23}]^{T}=(p_{13})_{*}(p_{12}^{-1}(a_1,a_2)\cap p_{23}^{-1}(a_2,a_3))=$$$$(p_{13})_{*}
 \{(a_1,a_2,x_3)|x_3\in M_3)\cap (x_1,a_2,a_3)|x_1\in M_1 \}=$$$$
 (p_{13})_{*}(Eu_{T}(x,\mathcal{N})[(a_1,a_2,a_3)]=
 (p_{13})_{*}(Eu_{T}(x,T_x(M_2))[(a_1,a_2,a_3)]=
 (Eu_{T}(x,T_x(M_2))[(a_1,a_3)]$$
 where
 $x=(a_1,a_2,a_3)$ and the last equality uses
 $$Eu(x,\mathcal{N})=T_x(M_1\times M_2\times M_3)/T_x(M_1)\oplus T_x( M_3)=T_x(M_2).$$

 \end{proof}
\begin{theorem}
 The action of $H^{G}_{*}(Z)$ on $H^{G}_{*}(Fl_{2n})$ is faithful. Namely, the map $H^{G}_{*}(Z)\xrightarrow[]{\rho} End(H^{G}_{*}(Fl_{2n}))$ given by convolution is injective.
 \end{theorem}
 \begin{proof}
 By diagram (3),
 we have the commutative diagram
 $$
 \xymatrix{
H^{G}_{*}(Z) \ar[d]^{r\circ i^*_{T,G}} \ar[r]^{\rho} &End(H^{G}_{*}(Fl_{n,d}))\ar[d]^{r\circ i^*_{T,G}\times r\circ i^*_{T,G}}\\
\cH^{T}_{*}(Z)\ar[r]^{\rho^{T}} &End(\cH^{T}_{*}(Fl_{n,d})).}$$
Since $r\circ i^*_{T,G}$ is injective, we only need to show that $\rho^{T}$ is injective.
Suppose there is an nonzero element $x\in \cH^{T}_{*}(Z)$ such that $\rho^{T}(x)=0$.
Since $Z=\bigsqcup Z_{\underline{\nu},\underline{\nu^{\prime}}}$, we have $\cH^{T}_{*}(Z)=\oplus \cH^{T}_{*}(Z_{\underline{\nu},\underline{\nu^{\prime}}})$, $x=(x_{\underline{\nu},\underline{\nu^{\prime}}})$.
So there exists $\underline{\nu},\underline{\nu^{\prime}}$ such that $x_{\underline{\nu},\underline{\nu^{\prime}}}\neq 0$.
Write $x_{\underline{\nu},\underline{\nu^{\prime}}}=\sum a_{w,w^{\prime}} p_{w,w^{\prime}}$. 
If for some $w,w^{\prime}$, $ a_{w,w^{\prime}}\neq 0$, take $p^{\prime}\in \cH^{T}(Fl_{2n})$,
 then the convolution product $a_{w,w^{\prime}} p_{w,w^{\prime}} \star p_{w^{\prime}}=a_{w,w^{\prime}} \Lambda_{w,w^{\prime}}p_w$. This term can not be eliminated since   
 $p_{w_1,w_1^{\prime}} \star p_{w^{\prime}}=0$ for $w_1^{\prime}\neq w^{\prime}$, $p_{w_1,w^{\prime}} \star p_{w^{\prime}}=\Lambda_{w_1,w^{\prime}} p_{w_1}$ and $\{p_{w}, w\in W \}$ are linearly independent in $\cH^{T}(Fl_{2n})$. So $a_{w,w^{\prime}}=0$ and $x$ must be $0$.
 \end{proof}

 Now we can state our main theorem of this section.
 Let $x_i$ be the characters of $T\subset Sp(V)$
coordinate to our chosen basis ${e_i}$ of $V$ so we have $S_{T}=k[x_1,\cdots,x_d]$.
 \begin{theorem} \label{action of type c}
 Fix a  partition $\underline{\nu}$.

 (a) For any $ 1 \leq i \leq n$ ,  $f \in  H_{*}^G(T^*Fl_{\underline{\nu} - 1_{i} })$, we have
 \begin{equation*}
\begin{split}
 \mathcal{F}_{i,r}(f)
  =  W_{P_{\underline{\nu}}}/(W_{P_{\underline{\nu}}} \cap W_{P_{\underline{\nu}-1_i}}) (\prod_{\bar{\nu}_{i-1} < t < \bar{\nu}_{i}}  (1 + \frac{ \hbar}{ x_{\bar{\nu}_i} - x_t})x_{\bar{\nu}_i}^r\cdot f).
 \end{split}
 \end{equation*}

 (b) If $ 1 \leq i < n$ ,  for any  $f \in  H_{*}^G(T^*Fl_{\underline{\nu} + 1_{i} })$, we have
 \begin{equation*}
\begin{split}
 \mathcal{E}_{i,r}(f)
  =   W_{P_{\underline{\nu}}}/(W_{P_{\underline{\nu}}} \cap W_{P_{\underline{\nu}+1_i}}) (\prod_{\bar{\nu}_{i} + 1 < t \leq  \bar{\nu}_{i+1}} (1 + \frac{ \hbar}{ x_{\bar{\nu}_i+1} - x_t})x_{\bar{\nu}_i+1}^r\cdot f).
 \end{split}
 \end{equation*}

 (c)  If $i = n$,  for any  $f \in  H_{*}^G(T^*Fl_{\underline{\nu} + 1_{n} })$  , we have
 \begin{equation*}
\begin{split}
 & \mathcal{E}_{i,n}(f)\\
  = &    W_{P_{\underline{\nu}}}/(W_{P_{\underline{\nu}}} \cap W_{P_{\underline{\nu}+1_n}})\left(\prod_{\bar{\nu}_{n} + 1 < t \leq  d}
  (1 + \frac{ \hbar}{ x_{\bar{\nu}_n+1} - x_t}) (1 + \frac{ \hbar}{ x_{\bar{\nu}_n+1} + x_t})
  (1 + \frac{ \hbar}{ 2x_{\bar{\nu}_n+1}})x_{\bar{\nu}_n+1}^r \cdot f\right).
 \end{split}
 \end{equation*}
 \end{theorem}

 \begin{proof}
 
We now compute the action of $[Z^e_{\underline{\nu},\underline{\nu}+1_i}]$ on $H_{*}(Fl_{\underline{\nu}+1_i})$.
To study the $T$-fixed point, we project
$Z^e_{\underline{\nu},\underline{{\nu}^{\prime}}}=\{gP_{\underline{\nu}},gP_{\underline{\nu}}P_{\underline{\nu}^{\prime}}/P_{\underline{\nu}^{\prime}}), g\in G\}$
 to the first factor. Since $(G/P_{\underline{\nu}})^T=W/W_{P_{\underline{\nu}}}$ and for $g$ a lift of $\bar{w}\in W/W_{P_{\underline{\nu}}}$, $(gP_{\underline{\nu}}P_{\underline{\nu}^{\prime}}/P_{\underline{\nu}^{\prime}})^{T}=wW_{P_{\underline{\nu}}}W_{P_{\underline{\nu}^{\prime}}}/W_{P_{\underline{\nu}^{\prime}}}$.
 We apply $r\circ i^{*}_{T,G}$  to
 $c_1(L^{+}_i)^r[Z^e_{\underline{\nu},\underline{{\nu}}+1_i}]$. By Theorem 3 it is
 $$\sum \limits_{w\in W/W_{P_{\underline{\nu}}},w^{\prime}\in wW_{P_{\underline{\nu}}}W_{P_{\underline{\nu}'}}} ww^{\prime}(x_{\overline{\nu}_i+1}) (\Lambda_{w,w^{\prime}})^{-1}p_{w,w^{\prime}},$$ where we use the fact that $c_1(L^{+}_i)$ on the component $Z_{\underline{\nu},\underline{\nu}+1_i}$ is $x_{\overline{v_i}+1}$ and
 $\Lambda_{w,w^{\prime}}$ is the $T\times C^{*}$-equivariant Euler class of $Z^e_{\underline{\nu},\underline{{\nu}}+1_i}$ at point $(w,w^{\prime})$.

 Apply $r\circ i^{*}_{T,G}$ to $f\in H^{G}_{*}(Fl_{\underline{\nu}+1_i})$, we get
 $$\sum \limits_{w^{\prime}\in W/W_{P_{\underline{\nu}+1_i}}} w^{\prime}(f)\Lambda_{w^{\prime}}^{-1}p_{w^{\prime}},$$ where $\Lambda_{w^{\prime}}$ is the $T$-equivariant Euler class of $Fl_{\underline{\nu}+1_i}$ at point $w^{\prime}$.

Using theorem 4, we compute
 their convolution, which is
 \begin{align*}
 &\sum \limits_{w\in W/W_{P_{\underline{\nu}}}} 
 \sum \limits_{w^{\prime}\in wW_{P_{\underline{\nu}}}W_{P_{\underline{{\nu}}^{\prime}}}/W_{P_{\underline{\nu}^{\prime}}}} ww^{\prime}(x_{\overline{\nu_{i}}+1}^r)
 (\Lambda_{w,w^{\prime}})^{-1} w^{\prime}(f)\Lambda_{w^{\prime}}^{-1} p_{w,w^{\prime}}\star p_{w^{\prime}}\\
 =&
 \sum \limits_{w\in W/W_{P_{\underline{\nu}}}} 
 \sum \limits_{w^{\prime}\in wW_{P_{\underline{\nu}}}W_{P_{\underline{{\nu}}^{\prime}}}/W_{P_{\underline{\nu}^{\prime}}}} ww^{\prime}(x_{\overline{\nu_{i}}+1}^r)
 (\Lambda_{w,w^{\prime}})^{-1} w^{\prime}(f)\Lambda_{w^{\prime}}^{-1} \widetilde{\Lambda_{w^{\prime}}}
 p_{w},
 \end{align*}
 where
 $\widetilde{\Lambda_{w^{\prime}}}$ is the  equivariant Euler class in $T^{*}Fl_{\underline{\nu}+1}$ at point $p_{w^{\prime}}.$

 Suppose the convolution product of $c_1(L^{+}_i)[Z^{e}_{\underline{\nu},\underline{\nu}+1_i}]$ and $f\in H^{G}_{*}(Fl_{\underline{\nu}+1_i})$ is $g\in H^{G}_{*}(Fl_{\underline{\nu}})$. By the Theorem 3, we have the image of $g$ under $r \circ i^{*}_{T,G}$  is
 $$\sum\limits_{w\in W/W_{P_1}}  w(g)\Lambda_{w}^{-1} p_w.$$
 Compare the above two terms we get
 \begin{align*}
 g=&\sum\limits_{w^{\prime}\in W_{P_{\nu}}/W_{P_{{\nu},{\nu}^{\prime}}}}
 w^{\prime} (x_{\overline{\nu_i}+1}^r)
 (\Lambda_{e,w^{\prime}})^{-1} w^{\prime}(f)\Lambda_{w^{\prime}}^{-1} \widetilde{\Lambda_{w^{\prime}}}\Lambda_{e}\\
 =& \sum\limits_{w\in W_{P_{\nu}}/W_{P_{{\nu},{\nu}^{\prime}}}}
 w (x_{\overline{\nu_i}+1}^r)
 (\Lambda_{e,w})^{-1} w(f)\Lambda_{w}^{-1} \widetilde{\Lambda_{w}}\Lambda_{e}\\
 =&
 \sum\limits_{w\in W_{P_{\nu}}/W_{P_{{\nu},{\nu}^{\prime}}}}
 w (x_{\overline{\nu_i}+1}^r f)
 (\Lambda_{e,w})^{-1} \Lambda_{w}^{-1} \widetilde{\Lambda_{w}}\Lambda_{e}.
 \end{align*}
 Now we compute these Euler classes.
 Denote the set of roots in $\mathfrak{g}/\mathfrak{p}_{\underline{\nu}}$ (resp. $\mathfrak{g}/\mathfrak{p}_{\underline{\nu}+1_i}$) by $R^{-}_{\underline{\nu}}$  (resp. $R^{-}_{\underline{\nu}+1_i}$) and the corresponding positive roots by $R_{\underline{\nu}}$(resp. $R_{\underline{\nu}+1_i}$).
 The tangent space of $Z^{e}_{\underline{\nu},\underline{\nu}+1_i}$ at $(e,w)$ is the same as at the point $(e,e)$ but the $T$ action differs by conjugation of $w$. Similarly for others.
 So we first compute the action on the tangent space at origin.
 The tangent space  $$T(G/P_{\underline{\nu},\underline{\nu}+1_i})\oplus (T^{*}(G/P_{\underline{\nu}})\cap T^{*}(G/P_{\underline{\nu}+1_i}))$$ of $Z^{e}_{\underline{\nu},\underline{\nu}+1_i}$ at origin is
 $$\mathfrak{g}/\mathfrak{p}_{\underline{\nu},\underline{\nu}+1_i}\oplus
 (\mathfrak{g}/\mathfrak{p}_{\underline{\nu}})^{*}\cap (\mathfrak{g}/\mathfrak{p}_{\underline{\nu}+1_i})^{*},
$$
 where the first summand is the base part and the second summand is the bundle part.
 So we get $\Lambda_{e,e}=\Pi_{s\in R^{-}_{\underline{\nu}}\cup R^{-}_{\underline{\nu}+1_i}} \alpha_{s} \times \Pi_{t\in R_{\underline{\nu}}\cap R_{\underline{\mu}+1_i} } \alpha_t.$ The difference between 
 $\Lambda_e$ and $\widetilde{\Lambda_e}$ is the bundle part,  we have
 $\Lambda_{e}^{-1}\widetilde{\Lambda_e}$ is $\Pi_{s\in R_{\underline{\nu}+1_i}} \alpha_{s}$ and $\Lambda_e$ is $\Pi_{s\in  R^{-}_{\underline{\nu}}}\alpha_{s}$.
 So
the product of $\Lambda_{(e,w)}^{-1}\Lambda_{w}^{-1}\widetilde{\Lambda_{w}}\Lambda_{e}$ is 
 $$w\left(
\frac
{\Pi_{s\in R^{-}_{\underline{\nu}}} \alpha_{s} \times \Pi_{s\in R_{\underline{\nu}+1_i}} \alpha_{s}}
{
\Pi_{s \in R^{-}_{\underline{\nu}}\cup R^{-1}_{\underline{\nu}+1_i}} \alpha_{s} \times \Pi_{s\in R_{\underline{\nu }}\cap  R_{\underline{ \nu }+1_i }  }  \alpha_s }\right).$$
If $i < n$,  the formula is the same with the case of type A, so we only need to give the formula in the case $i =n$.
$$
\frac
{\Pi_{s\in R^{-}_{\underline{\nu}}} \alpha_{s} }
{
\Pi_{s \in R^{-}_{\underline{\nu}}\cup R^{-1}_{\underline{\nu}+1_i}} \alpha_{s} }
= \frac{1}{\prod_{\bar{\nu}_{n} + 1 < t \leq  d} (-e_{\bar{\nu}_n+1} + e_t) (-e_{\bar{\nu}_n+1} - e_t)
  (-2e_{\bar{\nu}_n+1})} .$$
 $$
\frac
{ \Pi_{s\in R_{\underline{\nu}+1_i}} \alpha_{s}}
{ \Pi_{s\in R_{\underline{\nu }}\cap  R_{\underline{ \nu }+1_i }  }  \alpha_s }
=\prod_{\bar{\nu}_{n} + 1 < t \leq  d} (e_{\bar{\nu}_n+1} - e_t + \hbar) (e_{\bar{\nu}_n+1} + e_t + \hbar)
  (2e_{\bar{\nu}_n+1} + \hbar) .
$$
Then the theorem follows.
\end{proof}
\begin{theorem}
The assignment
$$\mathbf{e}_{i,r} \mapsto \mathcal{E}_{i,r}, \mathbf{f}_{i,r} \mapsto \mathcal{F}_{i,r},\mathbf{h}_{i,r} \mapsto \mathcal{H}_{i,r}$$
extends to an algebra homomorphism $\widetilde{Y}\rightarrow H^{G}_{*}(Z)$.
\end{theorem}
\begin{proof}
We only need to prove the relation $R1-R6$, the other relations of the $\widetilde{Y}$ are straightforward.

$\text{Relation} \ (R1)$.  Let $(fl, fl^{\prime},x)$  be an element in $Z^{e}_{\underline{\nu},\underline{\nu}+1_{j}}$. Recall
$$F_{i,\underline{\nu}} = -(1+q^{-2})V_i + q^{-1}(V_{i+1} +V_{i-1}),$$
$$F_{i,\underline{\nu} + 1_j} = -(1+q^{-2})V'_i + q^{-1}(V'_{i+1} + V'_{i-1}).$$
Then we have
$$F_{i,\underline{\nu}} - q^2F_{i,\underline{\nu}} = F_{i,\underline{\nu} + 1_j} - q^2F_{i,\underline{\nu} + 1_j} + ([2\delta_{ij} - \delta_{|i-j|,1}] -\delta_{in}\delta_{jn})(q^{-1} -q)L_j^{+},$$
where $[n]=\frac{q^n-q^{-n}}{q-q^{-1}}$.
It follows that $[\mathcal{H}_{i,r},\mathcal{E}_{j,s}]$ is the coefficient of $\hbar z^{-r-1}$ in
 \begin{align*}
 &(\lambda_{-\frac{1}{z}}(F_{i,\underline{\nu}} - q^2F_{i,\underline{\nu}}) - \lambda_{-\frac{1}{z}}(F_{i,\underline{\nu} + 1_j} - q^2F_{i,\underline{\nu} + 1_j}))^{-}\mathcal{E}_{j,s}\\
 &=\left(\lambda_{-\frac{1}{z}}(([2\delta_{ij} - \delta_{|i-j|,1}] -\delta_{in}\delta_{jn})(q^{-1} -q)L_j^{+}) - 1 \right)\lambda_{-\frac{1}{z}}(F_{i,\underline{\nu} + 1_j} - q^2F_{i,\underline{\nu} + 1_j}) \mathcal{E}_{j,s}.
 \end{align*}
We only give the proof for the case $i=j=n$, the other case is the same with \cite[Section 5]{varagnolo2000quiver}.
Let
$$A_s = \lambda_{-\frac{1}{z}}(F_{n}(\underline{\nu} + 1_n) - q^2F_{n}(\underline{\nu} + 1_n))\mathcal{E}_{n,s},$$
$$X = \lambda_{-\frac{1}{z}}(([2] -1)(q^{-1} -q)L_n^{+}) = \frac{(1-(c_n^+ - \hbar)z^{-1})(1-(c_n^+ + \frac{\hbar}{2})z^{-1})}{(1-(c_n^+ + \hbar)z^{-1})(1-(c_n^+ - \frac{\hbar}{2})z^{-1})},$$
where $c_n^+ = c_1{L_n^{+}}.$
Then the LHS and the RHS of the relation (R1) are respectively equal to the coefficient of $\hbar z^{-r+1}$ in
 \begin{align*}
 &(2(X-1)A_{s+2} - \hbar^2(X-1)A_s - 4z(X-1)A_{s+1} + 2z^2(X-1)A_s)^-\\
 &=(2(c_n^+)^2(X-1)A_{s} - \hbar^2(X-1)A_s - 4zc_{n}^+(X-1)A_{s} + 2z^2(X-1)A_s)^-
 \end{align*}
and
$$\hbar z(X+1)A_s - \hbar (X+1)A_{s+1} = \hbar z(X+1)A_s - \hbar c_n^+(X+1)A_{s}.$$
By direct computation, we have
 \begin{align*}
 &2(c_n^+)^2(X-1)A_{s} - \hbar^2(X-1)A_s - 4zc_{n}^+(X-1)A_{s} + 2z^2(X-1)A_s\\
 &=\hbar z(X+1)A_s - \hbar c_n^+(X+1)A_{s}.
 \end{align*}
The relation (R1) follows. The relation(R2) can be proved similarly.

$\text{Relation} \ (R3-R4)$.  The action of $\mathcal{F}_{i,r}$ and $\mathcal{F}_{j,r}$ for $j \neq n$  is the same as the  type A case in \cite{varagnolo2000quiver}. So the relation (R4) and the relation (R3) for $i \neq n, j \neq n$ follows. It is clear that (R3) holds for the case $i=n, |i-j|>1$. We only give the proof the relation (R3) for $i=n-1,\ j=n$.
We want to compute 
$\mathcal{E}_{i,r}\star \mathcal{E}_{j,s}$ and $\mathcal{E}_{j,r}\star \mathcal{E}_{i,s}$.
Let us consider   $\widetilde{Z}_{ij} = p_{12}^{-1}Z^{e}_{\underline{\nu} ,\underline{\nu} + 1_i} \cap p_{23}^{-1}Z^{e}_{\underline{\nu} + 1_i,\underline{\nu} + 1_i + 1_j}$
(resp. $\widetilde{Z}_{ji} = p_{12}^{-1}Z^{e}_{\underline{\nu} ,\underline{\nu} + 1_j} \cap p_{23}^{-1}Z^{e}_{\underline{\nu} + 1_j,\underline{\nu} + 1_i + 1_j}$)
in  $T^*Fl_{\underline{\nu}} \times  T^*Fl_{\underline{\nu} + 1_i}  \times T^*Fl_{\underline{\nu} + 1_i +1_j}$
(resp. $T^*Fl_{\underline{\nu}} \times  T^*Fl_{\underline{\nu} + 1_j}  \times T^*Fl_{\underline{\nu} + 1_i +1_j}$ ).

Let $d_1 = \text{dim} Fl_{\underline{\nu}}, 
d_2 = \text{dim} Fl_{\underline{\nu} + 1_i},
d_3 = \text{dim} Fl_{\underline{\nu} + 1_i +1_j}$.
then we have
 \begin{align*}
 &\text{dim} T^*Fl_{\underline{\nu}} \times  T^*Fl_{\underline{\nu} + 1_i}  \times T^*Fl_{\underline{\nu} + 1_i +1_j} = 2d_1 +2d_2 +2d_3,\\
&\text{dim}p_{12}^{-1}Z^{e}_{\underline{\nu} ,\underline{\nu} + 1_i} = d_1 + d_2 +2d_3 , \\
&\text{dim}p_{23}^{-1}Z^{e}_{\underline{\nu} + 1_i,\underline{\nu} + 1_i + 1_j} = 2d_1 + d_2 +d_3.
 \end{align*}
Moreover, for any $(x,y,z) \in p_{12}^{-1}Z^{e}_{\underline{\nu} ,\underline{\nu} + 1_i} \cap p_{23}^{-1}Z^{e}_{\underline{\nu} + 1_i,\underline{\nu} + 1_i + 1_j}$, $y$ is uniquely determined by $(x,z)$, that is, the restriction of $p_{13}$ to $\widetilde{Z}_{ij}$ is an isomorphism.  So
  \begin{align*}
 & \text{dim}p_{12}^{-1}Z^{e}_{\underline{\nu} ,\underline{\nu} + 1_i} + \text{dim}p_{23}^{-1}Z^{e}_{\underline{\nu} + 1_i,\underline{\nu} + 1_i + 1_j} - \text{dim} T^*Fl_{\underline{\nu}} \times  T^*Fl_{\underline{\nu} + 1_i}  \times T^*Fl_{\underline{\nu} + 1_i +1_j}\\
&= d_1 + d_3  , \\
&=\text{dim} \widetilde{Z}_{ij}.
 \end{align*}
i.e. the above  intersection is transverse.

If $j =i + 1$,
  Consider 
$$B_{i,i+1} : V_i'/V_i \hookrightarrow V_{i+1}'/V_{i+1},$$
as a section of the vector bundle $ L(q)\otimes( V_i'/V_i)^{\ast}\otimes V_{i+1}'/V_{i+1}  $  over $\widetilde{Z}_{ji}$.
 Let us denote it by $s_{ji}$ .
 Similarly, $x: V_{i+1}'/V_{i+1}  \rightarrow V_i'/V_i$ is a section (denoted by $s_{ij}$) of the vector bundle
$L(q) \otimes (V_{i+1}'/V_{i+1})^{\ast} )\otimes V_{i}'/V_{i} $ over $\widetilde{Z}_{ij}$.

The section $s_{ji}$ and  $\widetilde{Z}_{ji}$ are codimension one in the bundle $L(q)\otimes (V_i'/V_i)^{\ast} \otimes V_{i+1}'/V_{i+1}  $ . if $i = n-1$,  we have
  $$ \widetilde{Z}_{ji} = \{(fl,fl',x) | V_{n} \subseteq V_{n}' \subseteq V_n'^{\perp},V_{n-1} \subseteq V_{n-1}' \subseteq V_n', x(V_n') =V_{n-1},x(V_{n-1}')=0\}$$
$$\widetilde{Z}_{ij} = \{(fl,fl',x) | V_{n-1} \subseteq V_{n-1}' \subseteq V_n,V_{n} \subseteq V_{n}' \subseteq V_n'^{\perp}, x(V_n') =V_{n-1}',x(V_{n-1}')=0\}$$
  \begin{align*}
 & \widetilde{Z}_{ji} \cap s_{ji} = \widetilde{Z}_{ij} \cap s_{ij}\\
&=  \{(fl,fl',x) | V_{n} \subseteq V_{n}' \subseteq V_n'^{\perp},V_{n-1} \subseteq V_{n-1}' \subseteq V_n, x(V_n') =V_{n-1},x(V_{n-1}')=0\}.
 \end{align*}
 Thus $\widetilde{Z}_{ji} \cap s_{ji} = \widetilde{Z}_{ij} \cap s_{ij}$  is codimension two in $L(q)\otimes (V_i'/V_i)^{*}\otimes V_{i+1}'/V_{i+1}  $ over $\widetilde{Z}_{ji}$. Therefore, $s_{ij}$ (resp. $s_{ij}$) are transversal to zero section. 
  Then
 $$c_1(L(q)\otimes( V_i'/V_i)^{\ast}\otimes V_{i+1}'/V_{i+1})\mathcal{E}_{i+1,r}\ast \mathcal{E}_{i,s} = c_1(L(q)\otimes (V_{i+1}'/V_{i+1})^{\ast} )\otimes V_{i}'/V_{i})\mathcal{E}_{i,s}\ast \mathcal{E}_{i+1,r}, $$
 i.e.
 $$(c_{n} - c_{n-1} + \frac{\hbar}{2})\mathcal{E}_{i+1,r}\ast \mathcal{E}_{i,s} = (c_{n - 1} - c_{n} + \frac{\hbar}{2})\mathcal{E}_{i,s}\ast \mathcal{E}_{i+1,r}.$$
 The relation (R4) follows from this.

$\text{Relation} \ (R5-R6).$ In order to prove the last two relations, it suffices to prove the case $n = 1$.
Fix  a partition $I = (0,i,2d-2i,i)$, and  $Q  \in (S_T\otimes \CC)^{W_{P_I}}.$  Let $Q \otimes e_i \in H^G_{*}(fl_I) $. Set
$$\omega(\pm l, \pm k) = 1+ \frac{\hbar}{\mp x_l -(\mp x_k)}, \omega(\pm l, \mp k) = 1+ \frac{\hbar}{\mp x_l -(\pm x_k)}, \forall l,k \in [1,d].$$
 
Then we have
$$\mathbf{h}_{1,0}(Q\otimes e_i) = (2d - 3i)Q \otimes e_i.$$
 \begin{align*}
 \mathbf{f}_{1,0}(Q\otimes e_i)
 &= \sum\limits_{l = 1}^{1+i} (l,1+i)(\prod\limits_{s = 1}^i  (1+ \frac{\hbar}{ x_s - x_{1+i}})Q)\otimes e_{i+1}\\
 &=\sum\limits_{l = 1}^{1+i}\prod\limits_{s = 1,s\neq l}^i \omega(l,s)(l,i+1)Q \otimes e_{i+1}.
 \end{align*}

 \begin{align*}
 \mathbf{e}_{1,0}(Q\otimes e_i)
 &= \sum\limits_{k = i}^d \mathbb{Z}_{2,k}\times (k,i)[(\prod\limits_{t = i+1}^d (1+ \frac{\hbar}{ x_i - x_{t}}) (1+ \frac{\hbar}{ x_i + x_{t}})(1+ \frac{\hbar}{2x_i})Q]\otimes e_{i-1}.\\
 &=\sum\limits_{k = i}^d\prod\limits_{t = i,t\neq k}^d\omega(t,k)\omega(-t,k)\omega(-k,k)(k,i)Q \otimes e_{i-1}\\
 & \quad +\sum\limits_{k = i}^d\prod\limits_{t = i,t\neq k}^d\omega(t,-k)\omega(-t,-k)\omega(k,-k)[r]_k(k,i)Q \otimes e_{i-1} .
 \end{align*}
  Then the remaining calculation is the same with that of Appendix in the paper\cite{fan2019equivariant} so we skip it and
 the proposition follows.

\end{proof}
\section{ Type B case }

Let $V$ be a $2d+1$ dimensional  vector space  over complex numbers.
We still fix a nondegenerate skew-symmetric bilinear form  $(\ ,\ )$. For any subspace $W\subset V$,  let $W^{\perp}=\{x\in V|(x,w)=0 ~\text{for any}~ w\in W\}$.
We can define the flag variety of type B as follows,
$$Fl_B=\{ fl=( 0=V_0 \subset V_1  \cdots \subset V_{2n+1}=V)\mid  V_i=V_{2n+1-i}^{\perp}  \}. $$

For a flag $fl \in Fl_B$, we can still define  $\nu_{i}=dim V_{i}-dim V_{i-1}$ and $\underline{\nu}=(\nu_1,\cdots , \nu_{2n+1})$ be the dimension vector.
We denote by $Fl_{B,\underline{\nu}}$ the flag variety in $V$ of dimension vector $\underline{\nu}$.
Namely,
$$Fl_{B,\nu}=\{0\subset V_{1}\subset V_{2} \cdots V_{n}\subset V_{n+1} \subset \cdots \subset V_{2n}\subset V \ | \  dim V_{i}-dim V_{i-1}=\nu_{i}\}.$$
Hence $Fl_{B}$ is the disjoint union of $Fl_{B,\underline{\nu}}$ over all dimension vectors $\overline{\nu}$.
Let $G_B$ be the  group $O(V)\times\CC^{*}$.
Similarly as section \ref{steinberg variety of type c} and Lemma \ref{orbit sturcture of type c},
we can define the Steinberg variety $Z_B$ of type B. The group $G_B$ still have a natural action on $ Z_B$ which has the following Bruhat decomposition
$$ Z_{B,\nu,\nu^{\prime}}=\bigcup_{{w\in  W_{P_{\nu}}} \backslash W/ W_{P_{\nu}^{\prime}} }Z^{w}_{B,\underline{\nu},\underline{\nu^{\prime}}}.$$
Define $\mathscr{E}_{i,r}$,  $\mathscr{F}_{i,r}$ ,$\mathscr{H}_{i,r}$ in the  same way as  $\mathcal{E}_{i,r}$, $\mathcal{F}_{i,r}$,
$$  \mathscr{E}_{i,r} = \sum\limits_{\underline{\nu}} (-1)^{\nu_{i+1}}c_1(L)\bullet[Z^{e}_{B,\underline{\nu},\underline{\nu}+1_{i}}]$$
$$  \mathscr{F}_{i,r} = \sum\limits_{\underline{\nu}} (-1)^{\nu_i}c_1(L)\bullet[Z^{e}_{B,\underline{\nu},\underline{\nu}-1_{i}}].$$

By a similar argument of Theorem \ref{action of type c}, we have the following theorem.
\begin{theorem}
 Fix a partition $\underline{\nu}$.

 (a) For any$ 1 \leq i \leq n$ ,  $f \in  H_{*}^G(T^*Fl_{\underline{\nu} - 1_{i} })$, we have
 \begin{equation*}
\begin{split}
 \mathscr{F}_{i,r}(f)
  =  W_{P_{\underline{\nu}}}/(W_{P_{\underline{\nu}}} \cap W_{P_{\underline{\nu} - 1_i}}) (\prod_{\bar{\nu}_{i-1} < t < \bar{\nu}_{i}}  (1 + \frac{ \hbar}{ x_{\bar{\nu}_i} - x_t})x_{\bar{\nu}_i}^r\cdot f).
 \end{split}
 \end{equation*}

 (b) If $ 1 \leq i < n$ ,  for any  $f \in  H_{*}^G(T^*Fl_{\underline{\nu} + 1_{i} })$, we have
 \begin{equation*}
\begin{split}
 \mathscr{E}_{i,r}(f)
  =   W_{P_{\underline{\nu}}}/(W_{P_{\underline{\nu}}} \cap W_{P_{\underline{\nu} + 1_i}}) (\prod_{\bar{\nu}_{i} + 1 < t \leq  \bar{\nu}_{i+1}} (1 + \frac{ \hbar}{ x_{\bar{\nu}_i+1} - x_t})x_{\bar{\nu}_i+1}^r\cdot f).
 \end{split}
 \end{equation*}

 (c)  If $i = n$,  for any  $f \in  H_{*}^G(T^*Fl_{\underline{\nu} + 1_{n} })$  , we have
 \begin{equation*}
\begin{split}
 & \mathscr{E}_{i,n,\underline{\nu}}(f)\\
  = &    W_{P_{\underline{\nu}}}/(W_{P_{\underline{\nu}}} \cap W_{P_{\underline{\nu} + 1_n}})((\prod_{\bar{\nu}_{n} + 1 < t \leq  d}
  (1 + \frac{ \hbar}{ x_{\bar{\nu}_n+1} - x_t}) (1 + \frac{ \hbar}{ x_{\bar{\nu}_n+1} + x_t}))
  (1 + \frac{ \hbar}{ x_{\bar{\nu}_n+1}})x_{\bar{\nu}_n+1}^r \cdot f).
 \end{split}
 \end{equation*}
 \end{theorem}

\begin{theorem}
The assignment
$$\mathbf{e}_{i,r} \mapsto \mathscr{E}_{i,r}, \mathbf{f}_{i,r} \mapsto \mathscr{F}_{i,r},\mathbf{h}_{i,r} \mapsto \mathscr{H}_{i,r}$$
can be extended to an algebra homomorphism $\widetilde{Y}\rightarrow H^{\mathbb{C}^{*} \times G}(Z)$.
\end{theorem}

\section{Twisted polynomial current Lie algebra}
\allowdisplaybreaks
It is well known that Yangian $Y_{\hbar}(\mathfrak{g})$ is a deformation of the universal enveloping algebra for the polynomial current Lie algebra $\mathfrak{g}[x]$.  Let $\sigma$ be the involution of $\mathfrak{g}$. The subalgebra
$$\mathfrak{g}[x]^{\sigma} = \{A(x) \in \mathfrak{g}[x] \mid \sigma(A(x)) = A(-x)\}$$
in the Lie algebra $\mathfrak{g}[x]$ is called the twisted polynomial current Lie algebra related to  $(\mathfrak{g}, \sigma)$.
Twisted Yangian $Y(\mathfrak{g},\sigma)$ is the deformation of the $U(\mathfrak{g}[x]^{\sigma})$. In the paper \cite{fan2019equivariant},  there is an equivariant K theory approach to quantum symmetric pair by using the Steinberg variety of type B/C. In the light of Molev's work about twisted Yangian, if  we  consider the symmetric pair $(\mathfrak{sl}_{2n+1}, \mathfrak{sl}_{2n+1}^{\theta})$ where $\theta$ is an involution of $\mathfrak{sl}_{2n+1}$ defined by $\theta(e_i) = f_{2n+1-i}, \theta(f_i) = e_{2n+1-i} $, we have the following special twisted polynomial current Lie algebra
$$\mathfrak{sl}_{2n+1}[x]^{\theta} = \{A(x) \in \mathfrak{sl}_{2n+1}[x] \mid \sigma(A(x)) = A(-x)\}.$$
By directly computation, we have the following proposition.
\begin{prop}
$U(\mathfrak{sl}_{2n+1}[x]^{\theta})$ is generated by the $e_{i,r},f_{i,r},h_{i,r}$  satisfied the following relations.
\begin{eqnarray*}
&&h_{i,r}h_{j,s}=h_{j,s}h_{i,r},\\
&&[h_{i,r}, e_{j,s}]=(2\delta_{i, j}\delta_{r,even} -\delta_{i,j+1}- \delta_{i,j-1}  + \delta_{i,n}\delta_{j,n})e_{j,r+s}, \\
&&[h_{i,r},f_{j,s}]=(-2\delta_{i, j}\delta_{r,even} +\delta_{i,j+1}+ \delta_{i,j-1}  - \delta_{i,n}\delta_{j,n})f_{j,r+s},\\
&& [e_{i,r}, f_{j,s}]=\delta_{ij}h_{i,r+s},  \quad \quad \,   {\rm if} \ i, j\neq n,\\
&&e_{i,r}e_{j,s}=e_{j,s}e_{i,r},\quad f_{i,r}f_{j,s}=f_{j,s}f_{i,r},\quad \hspace{2.1cm}  {\rm if}\ |i-j|>1,\\
&&[e_{i,r+1},e_{j,s}] = [e_{i,r},e_{j,s+1}],\\
&&[f_{i,r+1},f_{j,s}] = [f_{i,r},f_{j,s+1}],\\
&&[e_{i,r_1},[e_{i,r_2},e_{j,s}]]=0,\quad \hspace{.51cm}  {\rm if}\ |i-j|=1,\\
&&[f_{i,r_1},[f_{i,r_2},f_{j,s}]]=0,\quad \hspace{.71cm}  {\rm if}\ |i-j|=1,\\
&&[e_{n,r_1},[e_{n,r_2},f_{n,s}]]=-2(\delta_{r_1+s,even} + \delta_{r_2+s,even})e_{n,r_1+r_2+s}, \\
&& [f_{n,r_1},[f_{n,r_2},e_{n,s}]]=-2(\delta_{r_1+s,even} + \delta_{r_2+s,even})f_{n,r_1+r_2+s}.
\end{eqnarray*}
\end{prop}

Recall $Z$ (resp. $Z_B$)  is the Steinberg variety of type C (resp. B), there is a homomorphism from twisted Yangian $\widetilde{Y} $  to $H^{ G}(Z)$(resp. $H^{ G_B}(Z_B)$).  In the same spirit as Molev's twisted Yangian, so we expect that there exists twisted Yangian which  is the deformation of $U(\mathfrak{sl}_{2n+1}[x]^{\theta})$.

\begin{theorem}
There are algebra homomorphisms
$$U(\mathfrak{sl}_{2n+1}[x]^{\theta}) \rightarrow H^{Sp(2d)}(Z_C),$$
$$U(\mathfrak{sl}_{2n+1}[x]^{\theta}) \rightarrow H^{O(2d+1)}(Z_B).$$
\end{theorem}
\begin{proof}
We only prove the type C case, the type B case is similar.  we still consider the faithful representation of $H^{Sp(2d)}(Z)$ on $H^{Sp(2d)}(Fl_{2n})$.

We only need to prove the last two relations in the definition of current algebra, the others are straightforward. It suffices to prove the case $n = 1$.
 Using the same notation as  Theorem \ref{action of type c},
we have
\begin{align*}
 f_{1,r}(Q\otimes e_i)
 = &\sum\limits_{l = 1}^{1+i} (l,1+i)x_{1+i}^r Q\otimes e_{i+1}\\
 =&\sum\limits_{l = 1}^{1+i}x_{l}^r(l,i+1)Q \otimes e_{i+1},\\
 e_{1,r}(Q\otimes e_i)
 =& \sum\limits_{k = i}^d \mathbb{Z}_{2,k}\times (k,i)x_i^r\otimes e_{i-1}\\
 =&\sum\limits_{k = i}^d x_k^r (k,i)Q \otimes e_{i-1}
    +\sum\limits_{k = i}^d(-x_k)^r[1]_k(k,i)Q \otimes e_{i-1} .
 \end{align*}

We now show the penultimate identity. By a direct calculation, we have
 \begin{align*}
 &\frac{f_{1,r_1}f_{1,r_2}e_{1,s}}{2}(Q\otimes e_i)\\
 = &\sum_{l = 1}^{i}\sum_{k = i+1}^{d} (x_{i+1}^{r_1}x_l^{r_2}x_k^s + x_{i+1}^{r_2}x_l^{r_1}x_k^s)(l,k)Q\otimes e_{i+1}\\
  &+\sum_{l = 1}^{i}(x_{i+1}^{r_1}x_{l}^{r_2+s} + x_{i+1}^{r_2}x_{l}^{r_1+s})Q \otimes e_{i+1}\\
  &+ \sum_{l = 1}^{i}\sum_{k = i+1}^{d}( x_{i+1}^{r_1}x_{l}^{r_2}(-x_{k})^s + x_{i+1}^{r_2}x_{l}^{r_1}(-x_{k})^s)(l,k)[1]_lQ\otimes e_{i+1}\\
  &+\sum_{l = 1}^{i}(x_{i+1}^{r_1}x_{l}^{r_2}(-x_l)^s + x_{i+1}^{r_2}x_{l}^{r_1}(-x_l)^s)[1]_lQ \otimes e_{i+1}\\
   &+ \sum_{m=1}^i\sum_{l = 1,l \neq m}^{i}\sum_{k = i+2}^{d}x_m^{r_1}x_l^{r_2}x_k^s(m,i+1)(k,l)Q\otimes e_{i+1}\\
   &+ \sum_{m=1}^i\sum_{l = 1,l\neq m}^{i}( x_m^{r_1}x_l^{r_2}x_m^s + x_m^{r_2}x_l^{r_1}x_m^s)(l,i+1)Q\otimes e_{i+1}\\
  &+ \sum_{m=1}^i\sum_{l = 1,l \neq m}^{i}\sum_{k = i+2}^{d} x_m^{r_1}x_l^{r_2}(-x_k)^s(m,i+1)(k,l)[1]_{l}Q\otimes e_{i+1}\\
  &+ \sum_{m=1}^i\sum_{l = 1,l\neq m}^{i} (x_m^{r_1}x_l^{r_2}(-x_m)^s + x_m^{r_2}x_l^{r_1}(-x_m)^s )(l,i+1)[1]_mQ\otimes e_{i+1},
 \end{align*}

 \begin{align*}
 &e_{1,s}f_{1,r_2}e_{1,r_1}(Q\otimes e_i)\\
 =&\sum_{m=1}^i\sum_{l=1,l\neq m}^ix_{i+2}^sx_m^{r_2}x_{l}^{r_1}(m,i+2)(i+1,l)Q\otimes e_{i+1}\\
 &+\sum_{m=1}^ix_{i+2}^s(x_{m}^{r_1}x_{i+1}^{r_2} + x_m^{r_2}x_{i+1}^{r_1})(m,i+2)Q \otimes e_{i+1}\\
  &+\sum_{m=1}^{i+1}x_{i+2}^s(x_{m}^{r_1}x_{i+2}^{r_2} + x_m^{r_2}x_{i+2}^{r_1})(m,i+1)Q \otimes e_{i+1}\\
 &+\sum_{k = i+3}^d\sum_{m=1}^i\sum_{l=1,l \neq m}^ix_k^sx_m^{r_2}x_l^{r_1}(k,m)(i+1,l)Q \otimes e_{i+1}\\
 & + \sum_{k=i+3}^d\sum_{m=1}^ix_k^s(x_m^{r_1}x_{i+1}^{r_2} + x_m^{r_2}x_{i+1}^{r_1})(m,k)Q\\
 & + \sum_{k=i+3}^d\sum_{m=1}^{i+1}x_k^s(x_m^{r_1}x_{k}^{r_2} + x_m^{r_2}x_{k}^{r_1})(m,i+1)Q\\
 & + \sum_{m=1}^i\sum_{l=1,l\neq m}^i(- x_{i+2})^sx_m^{r_2}x_{l}^{r_1}(m,i+2)(i+1,l)[1]_mQ\otimes e_{i+1}\\
  & + \sum_{m=1}^i(-x_{i+2})^s(x_{m}^{r_1}x_{i+1}^{r_2} + x_m^{r_2}x_{i+1}^{r_1})(m,i+2)[1]_mQ \otimes e_{i+1}\\
  & + \sum_{m=1}^{i+1}(-x_{i+2})^s(x_{m}^{r_1}(-x_{i+2})^{r_2} + x_m^{r_2}(-x_{i+2})^{r_1})(m,i+1)Q \otimes e_{i+1}\\
  & + \sum_{k = i+3}^d\sum_{m=1}^i\sum_{l=1,l \neq m}^i(-x_k)^sx_m^{r_2}x_l^{r_1}(k,m)(i+1,l)[1]_mQ \otimes e_{i+1}\\
  & + \sum_{k=i+3}^d\sum_{m=1}^i(-x_k)^s(x_m^{r_1}x_{i+1}^{r_2} + x_m^{r_2}x_{i+1}^{r_1})(m,k)[1]_mQ\\
  & + \sum_{k=i+3}^d\sum_{m=1}^{i+1}(-x_k)^s(x_m^{r_1}(-x_{k})^{r_2} + x_m^{r_2}(-x_{k})^{r_1})(m,i+1)Q.
 \end{align*}

 \begin{align*}
&(f_{1,r_1}e_{1,s}f_{1,r_2}+f_{1,r_2}e_{1,s}f_{1,r_1})(Q\otimes e_{i})\\
=&\sum_{l=1}^i (x_{i+1}^{r_1+s}x_l^{r_2} + x_{i+1}^{r_2+s}x_l^{r_1})(l,i+1)Q \otimes e_{i+1}\\
& + \sum_{k= i+2}^d\sum_{l=1}^i (x_{i+1}^{r_1}x_k^sx_l^{r_2} + x_{i+1}^{r_2}x_k^sx_l^{r_1})(l,k)Q \otimes e_{i+1}\\
&+\sum_{k=i+1}^d (x_{i+1}^{r_1}x_{k}^{r_2+s} + x_{i+1}^{r_2}x_{k}^{r_1+s} )Q \otimes e_{i+1}\\
& + \sum_{m=1}^i\sum_{k= i+2}^d\sum_{l=1,l \neq m}^i (x_m^{r_1}x_k^sx_l^{r_2} + x_m^{r_2}x_k^sx_l^{r_1} )(m,i+1)(l,k)Q \otimes e_{i+1}\\
& + \sum_{m=1}^i\sum_{k= i+2}^d ( x_m^{r_1}x_k^sx_{i+1}^{r_2} + x_m^{r_2}x_k^sx_{i+1}^{r_1})(m,k)Q \otimes e_{i+1}\\
& + \sum_{m=1}^i\sum_{l= 1,l \neq m}^i(x_m^{r_1+s}x_l^{r_2} + x_m^{r_2+s}x_l^{r_1} )(l,i+1)Q \otimes e_{i+1}\\
& + \sum_{m=1}^i(x_m^{r_1+s}x_{i+1}^{r_2} + x_m^{r_2+s}x_{i+1}^{r_1})Q \otimes e_{i+1}\\
& + \sum_{m=1}^i\sum_{k= i+2}^d (x_m^{r_1}x_k^{r_2+s} + x_m^{r_2}x_k^{r_1+s})(m,i+1)Q \otimes e_{i+1}\\
& + \sum_{m=1}^i(x_{m}^{r_{1}+r_{2}+s} + x^{r_{1}+r_{2}+s})(m,i+1)Q \otimes e_{i+1}\\
& + \sum_{k = i+1}^d \sum_{l= 1}^i( x_{i+1}^{r_1}(-x_k)^sx_l^{r_2} +  x_{i+1}^{r_1}(-x_k)^sx_l^{r_2})(l,k)[1]_lQ \otimes e_{i+1}\\
 & + \sum_{k = i+1}^d (x_{i+1}^{r_1}(-x_k)^{r_2+s} + x_{i+1}^{r_2}(-x_k)^{r_1+s} )Q \otimes e_{i+1}\\
 & + \sum_{m = 1}^i \sum_{k = i+2}^d\sum_{l= 1, l \neq m}^i( x_m^{r_1}(-x_k)^sx_l^{r_2} + x_m^{r_2}(-x_k)^sx_l^{r_1})(m,i+1)(k,l)[1]_lQ \otimes e_{i+1}\\
 & + \sum_{m = 1}^i \sum_{k = i+2}^d (x_m^{r_1}(-x_k)^sx_{i+1}^{r_2} + x_m^{r_2}(-x_k)^sx_{i+1}^{r_1})(k,m)[1]_mQ \otimes e_{i+1}\\
& + \sum_{m = 1}^i \sum_{l= 1, l \neq m}^i (x_m^{r_1}(-x_m)^sx_l^{r_2} + x_m^{r_2}(-x_m)^sx_l^{r_1})(i+1,l)[1]_mQ \otimes e_{i+1}\\
& + \sum_{m = 1}^i  (x_m^{r_1}(-x_m)^sx_{i+1}^{r_2} + x_m^{r_2}(-x_m)^sx_{i+1}^{r_1})[1]_mQ \otimes e_{i+1}\\
& + \sum_{m = 1}^i  \sum_{k = i+2}^d (x_m^{r_1}(-x_k)^s(-x_k)^{r_2} + x_m^{r_2}(-x_k)^s(-x_k)^{r_1})(m,i+1)Q \otimes e_{i+1}\\
& + \sum_{m = 1}^i   (x_m^{r_1}(-x_m)^s(-x_m)^{r_2} + x_m^{r_2}(-x_m)^s(-x_m)^{r_1})(m,i+1)Q \otimes e_{i+1}\\
 \end{align*}
Therefore, we have
 \begin{align*}
&e_{1,r_1}e_{1,r_2}f_{1,s} + f_{1,s}e_{1,r_2}e_{1,r_1} - e_{1,r_1}f_{1,s}e_{1,r_2} - e_{1,r_2}f_{1,s}e_{1,r_1}\\
=&A Q\otimes e_{i+1} + \sum_{l=1}^iB_l (l,i+1)Q \otimes e_{i+1} \\
&+ \sum_{l=1}^i\sum_{k=i+2}^d C_{kl}(k,l)Q\otimes e_{i+1} + \sum_{m = 1}^i\sum_{l = 1,l < m}^i\sum_{k=i+2}^dD_{klm}(m,i+1)(k,l)Q\otimes e_{i+1} \\
&+ \sum_{l=1}^i\sum_{k=i+2}^d E_{kl}(k,l)[r]_lQ \otimes e_{i+1} +  \sum_{m = 1}^i\sum_{l = 1,l < m}^i\sum_{k=i+2}^dF_{klm}(m,i+1)(k,l)[1]_lQ\otimes e_{i+1}.
 \end{align*}
 By a directly computation, we have
 \begin{align*}
A=&  -2(\delta_{r_1 + s,even}+ \delta_{r_2 + s,even})x_{i+1}^{r_1 +r_2 + s},\\
B_l =& -2(\delta_{r_1 + s,even}+ \delta_{r_2 + s,even}) x_{l}^{r_1 +r_2 + s},\\
C_{kl}=&D_{klm} = E_{kl}=F_{klm}=0.
 \end{align*}

 That is,
 $$f_{1,r_1}f_{1,r_2}e_{1,s} + e_{1,s}f_{1,r_2}f_{1,r_1} - f_{1,r_1}e_{1,s}f_{1,r_2} - f_{1,r_2}e_{1,s}f_{1,r_1} =  -2(\delta_{r_1 + s,even}+ \delta_{r_2 + s,even})f_{1,r_1 +r_2 +s}.$$
 Similarly, we can prove another identity. The proposition follows.
\end{proof}
\bibliographystyle{alpha}
\bibliography{bib.tex}
\end{document}